\documentclass[12pt]{article}
\usepackage{amsmath,amsthm,amssymb}
\usepackage{amssymb,latexsym}
\newtheorem{theorem}{Theorem}

\newtheorem{lemma}{Lemma}

\begin{document}

\baselineskip=17pt

\title{\bf Prime numbers of the form $\mathbf{[n^c tan^\theta(log n)]}$}

\author{\bf S. I. Dimitrov}

\date{2021}

\maketitle

\begin{abstract}
Let $[\, \cdot\,]$ be the floor function.
In the present paper we prove that when $1<c<\frac{12}{11}$ and $\theta>1$ is a fixed,
then there exist infinitely many prime numbers of the form $[n^c \tan^\theta(\log n)]$.
\quad\\
{\bf  2020 Math.\ Subject Classification}:  11L07 $\cdot$ 11N25
\end{abstract}

\section{Notations}
\indent

The letter $p$  with or without subscript will always denote prime number.
By $\varepsilon$ we denote an arbitrary small positive constant, not the same in all appearances.
We denote by $\Lambda(n)$ von Mangoldt's function.
Moreover $e(y)=e^{2\pi i y}$. As usual $[t]$ and $\{t\}$ denote the integer part, respectively, the fractional part of $t$.
We denote by $\tau _k(n)$ the number of solutions of the equation $m_1m_2\ldots m_k$ $=n$ in natural numbers $m_1,\,\ldots,m_k$.
Throughout this paper we suppose that $1<c<\frac{12}{11}$.
Assume that $\theta>1$ is a fixed.
Denote
\begin{align}
\label{gamma}
&\gamma=\frac{1}{c}\,;\\
\label{psit}
&\psi(t)=\{t\}-1/2\,;\\
\label{Delta1}
&\Delta_1=e^{\pi\big[\frac{\log x}{\pi}\big]+\arctan1}\,;\\
\label{Delta2}
&\Delta_2=e^{\pi\big[\frac{\log x}{\pi}\big]+\arctan2}\,;\\
\label{Scx}
&S_c(x)=\sum\limits_{\Delta_1\leq n<\Delta_2\atop{[n^c \tan^\theta(\log n)]=p}}1\,.
\end{align}

\section{Introduction and statement of the result}
\indent

The problem of the existence of infinitely many  prime numbers of a special form is as
interesting as it is difficult in prime number theory.
It is not known if there exists any quadratic polynomial that takes infinitely many prime values.
There is a conjecture that there exist infinitely many  prime numbers of the form
$n^2+1$, which is out of reach of the current state of the mathematics.
For this reason, the mathematical world deals with the accessible problem of primes of the form $[n^c]$.

Let $\mathbb{P}$ denotes the set of all prime numbers.
In 1953 Piatetski-Shapiro \cite{Shapiro} has shown that for any fixed $1<c<\frac{12}{11}$ the set
\begin{equation*}
\mathbb{P}_c=\{p\in\mathbb{P}\;\;|\;\; p= [n^c]\;\; \mbox{ for some } n\in \mathbb{N}\}
\end{equation*}
is infinite.
The prime numbers of the form $p = [n^c]$ are called Piatetski-Shapiro primes.
Denote
\begin{equation*}
\pi_c(x)=\sum\limits_{n\leq x\atop{[n^c]=p}}1\,.
\end{equation*}
Piatetski-Shapiro's result states that
\begin{equation*}
\pi_c(x)\sim \frac{x}{c\log x}
\end{equation*}
for
\begin{equation*}
1<c<\frac{12}{11}\,.
\end{equation*}
Subsequently the interval for $c$ was sharpened many times \cite{Baker-Harman}, \cite{Heath2},
\cite{Jia1}, \cite{Jia2}, \cite{Kolesnik1}, \cite{Kolesnik2}, \cite{Kumchev}, \cite{Leitmann},  \cite{Liu-Rivat}, \cite{Rivat}.
To achieve a longer interval for $c$ the authors used the fact that the upper bound for $c$ is closely connected
with the estimate of an exponential sum over primes.
The best results up to now belongs to Rivat and Sargos \cite{Rivat-Sargos} with
\begin{equation*}
\pi_c(x)\sim \frac{x}{c\log x}
\end{equation*}
for
\begin{equation*}
1<c<\frac{2817 }{2426}
\end{equation*}
and to Rivat and Wu \cite{Rivat-Wu} with
\begin{equation*}
\pi_c(x)\gg \frac{x}{\log x}\,.
\end{equation*}
for $1<c<\frac{243}{205}$.

Motivated by these results we investigate the existence of infinitely many prime numbers of the form
associated with the form of Piatetski-Shapiro primes.
We show that when $1<c<\frac{12}{11}$ and $\theta>1$ is a fixed, then the set
\begin{equation*}
\mathbb{T}_c=\{p\in\mathbb{P}\;\;|\;\; p= [n^c \tan^\theta(\log n)]\;\; \mbox{ for some } n\in \mathbb{N}\}
\end{equation*}
is infinite.
More precisely we prove the following theorem.
\begin{theorem} Let $1<c<\frac{12}{11}$ and $\theta>1$ is a fixed.
Then
\begin{equation*}
S_c(x)\gg \frac{x}{\log x}\,.
\end{equation*}
\end{theorem}

\section{Lemmas}
\indent

\begin{lemma}\label{Squareoutlemma}
Let $I$ be a subinterval of $(X, 2X]$.   For any complex numbers $z(n)$ we have
\begin{equation*}
\bigg|\sum_{n\in I}z(n)\bigg|^2
\leq\bigg(1+\frac{X}{Q}\bigg)\sum_{|q|< Q}\bigg(1-\frac{|q|}{Q}\bigg)
\sum_{n,\, n+q\,\in I}z(n+q)\overline{z(n)}\,,
\end{equation*}
where $Q$ is any positive integer.
\end{lemma}
\begin{proof}
See (\cite{Heath2}, Lemma 5).
\end{proof}

\begin{lemma}\label{GrahamandKolesnik}
Let $k \geq0$ be an integer.
Suppose that $f(t)$ has $k+2$ continuous derivatives on $I$, and that $I \subseteq(N,2N]$.
Assume also that there is some constant $F$ such that
\begin{equation}\label{frFNR}
|f^{(r)}(t)|\asymp F N^{-r}
\end{equation}
for $r = 1, \ldots, k + 2$. Let $Q = 2^k$. Then
\begin{equation*}
\bigg|\sum_{n\in I}e(f(n))\bigg|\ll F^{\frac{1}{4Q-2}} N^{1-\frac{k+2}{4Q-2}}  +F^{-1} N\,.
\end{equation*}
The implied constant depends only upon the implied constants in \eqref{frFNR}.
\end{lemma}
\begin{proof}
See (\cite{Graham-Kolesnik}, Theorem 2.9).
\end{proof}

\begin{lemma}\label{Heath-Brown} Let $G(n)$ be a complex valued function.
Assume further that
\begin{align*}
&P>2\,,\quad P_1\le 2P\,,\quad  2\le U<V\le Z\le P\,,\\
&U^2\le Z\,,\quad 128UZ^2\le P_1\,,\quad 2^{18}P_1\le V^3\,.
\end{align*}
Then the sum
\begin{equation*}
\sum\limits_{P<n\le P_1}\Lambda(n)G(n)
\end{equation*}
can be decomposed into $O\Big(\log^6P\Big)$ sums, each of which is either of Type I
\begin{equation*}
\mathop{\sum\limits_{M<m\le M_1}a(m)\sum\limits_{L<l\le L_1}}_{P<ml\le P_1}G(ml)
\end{equation*}
and
\begin{equation*}
\mathop{\sum\limits_{M<m\le M_1}a(m)\sum\limits_{L<l\le L_1}}_{P<ml\le P_1}G(ml)\log l\,,
\end{equation*}
where
\begin{equation*}
L\ge Z\,,\quad M_1\le 2M\,,\quad L_1\le 2L\,,\quad a(m)\ll \tau _5(m)\log P
\end{equation*}
or of Type II
\begin{equation*}
\mathop{\sum\limits_{M<m\le M_1}a(m)\sum\limits_{L<l\le L_1}}_{P<ml\le P_1}b(l)G(ml)
\end{equation*}
where
\begin{equation*}
U\le L\le V\,,\quad M_1\le 2M\,,\quad L_1\le 2L\,,\quad
a(m)\ll \tau _5(m)\log P\,,\quad b(l)\ll \tau _5(l)\log P\,.
\end{equation*}
\end{lemma}
\begin{proof}
See (\cite{Heath1}).
\end{proof}

\begin{lemma}\label{Vaaler}
For every $M\geq2$, we have
\begin{equation*}
\psi(t)=\sum\limits_{1\leq|h|\leq M}a(h)e(ht)
+\mathcal{O}\Bigg(\sum\limits_{|h|\leq M}b(h)e(ht)\Bigg)\,,\quad
a(h)\ll1/|h|\,,\quad b(h)\ll1/M\,.
\end{equation*}
\end{lemma}
\begin{proof}
See \cite{Vaaler}.
\end{proof}

\section{Proof of the Theorem}
\indent

We have that
\begin{equation*}
[n^c \tan^\theta(\log n)]=p
\end{equation*}
if and only if
\begin{equation}\label{pncp+1}
p\leq n^c \tan^\theta(\log n)<p+1\,.
\end{equation}
From \eqref{Scx} and \eqref{pncp+1} we write
\begin{align}\label{ScxGammaSigma}
S_c(x)&=\sum\limits_{\Delta_1^c \tan^\theta(\log \Delta_1)-1<p< \Delta_2^c \tan^\theta(\log \Delta_2)}
\sum\limits_{\Delta_1\leq n<\Delta_2\atop{p\leq n^c \tan^\theta(\log n)<p+1}}1\nonumber\\
&=\sum\limits_{\Delta_1^c \tan^\theta(\log \Delta_1)-1<p< \Delta_2^c \tan^\theta(\log \Delta_2)}
\sum\limits_{m'_p\leq n<m''_p}1+\mathcal{O}(1)\nonumber\\
&=\sum\limits_{\Delta_1^c \tan^\theta(\log \Delta_1)-1<p<\Delta_2^c \tan^\theta(\log \Delta_2)}
\Big([-m'_p]-[-m''_p]\Big)+\mathcal{O}(1)\nonumber\\
&=\Gamma+\Sigma+\mathcal{O}(1)\,,
\end{align}
where
\begin{equation}\label{subset}
\big[m'_p, m''_p\big)\subset[\Delta_1, \Delta_2)
\end{equation}
and the interval $\big[m'_p, m''_p\big)$ is a solution of the system inequalities
\begin{equation}\label{System1}
\left|\begin{array}{cc}
n^c \tan^\theta(\log n)\geq p \quad\;\; \\
n^c \tan^\theta(\log n)<p+1
\end{array}\right.
\end{equation}
and  where
\begin{align}
\label{Gamma}
&\Gamma=\sum\limits_{\Delta_1^c \tan^\theta(\log \Delta_1)-1<p<\Delta_2^c \tan^\theta(\log \Delta_2)}
\big(m''_p-m'_p\big)\,,\\
\label{Sigma}
&\Sigma=\sum\limits_{\Delta_1^c \tan^\theta(\log \Delta_1)-1<p<\Delta_2^c \tan^\theta(\log \Delta_2)}
\big(\psi(-m''_p)-\psi(-m'_p)\big)\,.
\end{align}

\subsection{Estimation of $\mathbf{\Gamma}$}

Consider the function $t(y)$ defined by
\begin{equation}\label{Implicitfunction1}
t=y^c\tan^\theta(\log y)\,,
\end{equation}
for
\begin{equation}\label{yy'y''}
y\in\big[m'_p, m''_p\big]\,.
\end{equation}
The first derivative of $y$ as implicit function of $t$ is
\begin{equation}\label{Firstderivative}
y'=\frac{y^{1-c}}{\big(c\tan(\log y)+\theta\sec^2(\log y)\big)\tan^{\theta-1}(\log y)}\,.
\end{equation}
By \eqref{gamma}, \eqref{Delta1},  \eqref{Delta2},  \eqref{subset}, \eqref{System1},
\eqref{Gamma}, \eqref{Implicitfunction1}, \eqref{yy'y''} and the mean-value theorem  we obtain
\begin{equation}\label{Gammaest1}
\Gamma\gg\sum\limits_{\Delta_1^c \tan^\theta(\log \Delta_1)-1<p<\Delta_2^c \tan^\theta(\log \Delta_2)}y'\big(\xi_p\big)\,,
\end{equation}
where
\begin{equation}\label{System2}
\left|\begin{array}{cc}
p<\xi_p< p +1 \\
y(\xi_p)\asymp \xi_p^\gamma \quad  \quad
\end{array}\right..
\end{equation}
Using \eqref{gamma}, \eqref{Delta1},  \eqref{Delta2},  \eqref{subset}, \eqref{yy'y''}, \eqref{Firstderivative},
\eqref{Gammaest1}, \eqref{System2} and the prime number theorem we deduce
\begin{equation}\label{Gammaest2}
\Gamma\gg\sum\limits_{\Delta_1^c \tan^\theta(\log \Delta_1)-1<p<\Delta_2^c \tan^\theta(\log \Delta_2)} p^{\gamma-1}\gg \frac{x}{\log x}\,.
\end{equation}

\subsection{Estimation of $\mathbf{\Sigma}$}

Using \eqref{Sigma} and  Abel's summation formula in a standard way we get
\begin{equation}\label{Sigmaest1}
\Sigma\ll \frac{x}{\log^2x}+\max_{\Delta_1^c \tan^\theta(\log \Delta_1)-1 \leq t\leq \Delta_2^c \tan^\theta(\log \Delta_2)} |\Sigma_1(t)|\,,
\end{equation}
where
\begin{equation}\label{Sigma1}
\Sigma_1(t)=\sum\limits_{\Delta_1^c \tan^\theta(\log \Delta_1)-1<n\leq t}\Lambda(n) \Psi(n)
\end{equation}
and where
\begin{equation}\label{Psipsi}
\Psi(n)=\psi(-m''_n)-\psi(-m'_n)\,.
\end{equation}
Splitting the range of $n$ into dyadic subintervals from \eqref{Sigma1} we obtain
\begin{equation}\label{Sigma1est1}
\Sigma_1(t)\ll (\log X)|\Sigma_2(N)|\,,
\end{equation}
where
\begin{equation}\label{Sigma2}
\Sigma_2(N)=\sum\limits_{N<n\leq 2N}\Lambda(n) \Psi(n)
\end{equation}
and where
\begin{equation}\label{Nlimits}
\Delta_1^c \tan^\theta(\log \Delta_1)-1\leq N\leq\frac{t}{2}\,.
\end{equation}
Having in mind \eqref{Sigmaest1}, \eqref{Sigma1est1} and \eqref{Nlimits} we deduce
\begin{equation}\label{Sigmaest2}
\Sigma\ll \frac{x}{\log^2x}+(\log x)\max_{\Delta_1^c \tan^\theta(\log \Delta_1)-1
\leq N\leq \frac{1}{2}\Delta_2^c \tan^\theta(\log \Delta_2)} |\Sigma_2(N)|\,.
\end{equation}
Hence fort we shall use that
\begin{equation}\label{NXc}
\left|\begin{array}{cc}
N\asymp x^c \hspace{68mm}\\
\Delta_1^c \tan^\theta(\log \Delta_1)-1\leq N\leq \frac{1}{2}\Delta_2^c \tan^\theta(\log \Delta_2)
\end{array}\right..
\end{equation}
It remains to find a non-trivial estimate for $\Sigma_2(N)$.

Let $M\geq2$ is a real parameter, we shall choose latter depending on $N$.
From \eqref{Psipsi} and  Lemma \ref{Vaaler} it follows
\begin{equation}\label{Psisum}
\Psi(n)=\sum\limits_{1\leq|h|\leq M}a(h)\omega(n,h)
+\mathcal{O}\big(\omega_1(n)\big)+\mathcal{O}\big(\omega_2(n)\big)\,,
\end{equation}
where
\begin{align}
\label{omega}
&\omega(n,h)=e(-hm''_n)-e(-hm'_n)\,,\\
\label{omega1}
&\omega_1(n)=\sum\limits_{|h|\leq M}b(h)e(-hm'_n)\,,\\
\label{omega2}
&\omega_2(n)=\sum\limits_{|h|\leq M}b(h)e(-hm''_n)\,,\\
\label{ahbh}
&a(h)\ll1/|h|\,,\quad b(h)\ll1/M\,.
\end{align}
Now \eqref{Sigma2}, \eqref{Psisum}, \eqref{omega1} and \eqref{omega2}  and yields
\begin{equation}\label{Sigma2est1}
\Sigma_2(N)=\sum\limits_{1\leq|h|\leq M}a(h)\sum\limits_{N<n\leq 2N}\Lambda(n)\omega(n,h)
+\mathcal{O}\big(\Omega_1\log N\big)+\mathcal{O}\big(\Omega_2\log N\big)\,,
\end{equation}
where
\begin{align}
\label{Omega1}
&\Omega_1=\sum\limits_{N<n\leq 2N}\sum\limits_{|h|\leq M}b(h)e(-hm'_n)\,,\\
\label{Omega2}
&\Omega_2=\sum\limits_{N<n\leq 2N}\sum\limits_{|h|\leq M}b(h)e(-hm''_n)\,.
\end{align}
We first estimate $\Omega_1$.
According to \eqref{subset}, \eqref{System1}, \eqref{NXc} and \eqref{Omega1} the variable $m'_n$ is an implicit function of $n$ defined by
\begin{equation}\label{Implicitfunction2}
y^c\tan^\theta(\log y)=n \quad \mbox{ for } \quad  n\in (N, 2N]
\end{equation}
and
\begin{equation}\label{yDelta1Delta2}
y\subset[\Delta_1, \Delta_2)\,.
\end{equation}
Taking into account  \eqref{gamma}, \eqref{Delta1},  \eqref{Delta2}, \eqref{Implicitfunction1}, \eqref{Firstderivative}, \eqref{NXc},
\eqref{Implicitfunction2} and \eqref{yDelta1Delta2} we find
\begin{equation}\label{firstderivativenasymp}
y'\asymp  N^{\gamma-1}\,.
\end{equation}
Proceeding in the same way we get
\begin{equation}\label{secondderivativenasymp}
y''\asymp  N^{\gamma-2}\,.
\end{equation}
Now \eqref{NXc}, \eqref{ahbh}, \eqref{Omega1}, \eqref{firstderivativenasymp}, \eqref{secondderivativenasymp}
and Lemma \ref{GrahamandKolesnik} with  $k=0$  imply
\begin{align}\label{Omega1est1}
\Omega_1&\ll \frac{1}{M}\sum\limits_{|h|\leq M}\bigg|\sum\limits_{N<n\leq 2N} e(-hm'_n)\bigg|\nonumber\\
&\ll \frac{N}{M}+\frac{1}{M}\sum\limits_{1\leq h\leq M}\bigg|\sum\limits_{N<n\leq 2N} e(-hm'_n)\bigg|\nonumber\\
&\ll \frac{N}{M}+\frac{1}{M}\sum\limits_{1\leq h\leq M}\Big( h^\frac{1}{2}N^\frac{\gamma}{2} +  h^{-1}N^{1-\gamma} \Big)\nonumber\\
&\ll NM^{-1}+  N^\frac{\gamma}{2}M^\frac{1}{2} + N^{1-\gamma} M^{-1} \log M\,.
\end{align}
Take
\begin{equation}\label{M}
M=N^{1-\gamma} \log^4 N\,.
\end{equation}
By \eqref{Omega1est1} and \eqref{M} we obtain
\begin{equation*}
\Omega_1\ll \frac{N^\gamma}{\log^4 N}+ N^\frac{1}{2}\log^2 N\,.
\end{equation*}
The last estimate and \eqref{NXc} give us
\begin{equation}\label{Omega1est2}
\Omega_1\ll \frac{x}{\log^4 x}+ x^\frac{c}{2}\log^2 x\ll  \frac{x}{\log^4 x}\,.
\end{equation}
Arguing in the same way we get
\begin{equation}\label{Omega2est1}
\Omega_2\ll\frac{x}{\log^4 x}\,.
\end{equation}
From \eqref{NXc}, \eqref{Sigma2est1}, \eqref{Omega1est2} and  \eqref{Omega2est1} it follows
\begin{equation}\label{Sigma2est2}
\Sigma_2(N)\ll|\Sigma_3(N)|+\frac{x}{\log^3 x}\,,
\end{equation}
where
\begin{equation}\label{Sigma3}
\Sigma_3(N)=\sum\limits_{1\leq|h|\leq M}a(h)\sum\limits_{N<n\leq 2N}\Lambda(n)\omega(n,h)\,.
\end{equation}
By \eqref{omega} we have
\begin{equation}\label{omegaPhi}
\omega(n,h)=\Phi_h(n)e(-hm'_n)\,,
\end{equation}
where
\begin{equation}\label{Phi}
\Phi_h(n)=e\big(h(m'_n-m''_n)\big)-1\,.
\end{equation}
Using \eqref{omegaPhi}, \eqref{Phi} and  Abel's summation formula we find
\begin{align}\label{Lambdaomega}
\sum\limits_{N<n\leq 2N}\Lambda(n)\omega(n,h)&=\Phi_h(2N)\sum\limits_{N<n\leq 2N}\Lambda(n)e(-hm'_n)\nonumber\\
&-\int\limits_{N}^{2N}\Phi'_h(t)\sum\limits_{N<n\leq t}\Lambda(n)e(-hm'_n)\,dt\nonumber\\
&\ll\Bigg(|\Phi_h(2N)|+\int\limits_{N}^{2N}|\Phi'_h(t)|\,dt\Bigg)
\max\limits_{N_1\in[N, 2N]}|\Sigma_4(N, N_1)|\,,
\end{align}
where
\begin{equation}\label{Sigma4}
\Sigma_4(N, N_1)=\sum\limits_{N<n\leq N_1}\Lambda(n)e(hm'_n)\,.
\end{equation}
Consider $\Phi_h(2N)$. By \eqref{subset}, \eqref{System1}, \eqref{yy'y''}, \eqref{Firstderivative},
\eqref{NXc}, \eqref{Implicitfunction2}, \eqref{yDelta1Delta2}, \eqref{Phi} and the mean-value theorem we get
\begin{equation}\label{Phiest1}
|\Phi_h(2N)|\leq2\big|\sin\big(\pi h(m'_n-m''_n)\big)\big|\ll
|h|\big|m'_n-m''_n\big|\ll|h|N^{\gamma-1}\,.
\end{equation}
On the other hand \eqref{secondderivativenasymp}   and  \eqref{Phi} yields
\begin{equation}\label{Phiest2}
\Phi'_h(t)\ll|h|N^{\gamma-2}  \quad\mbox{ for }\quad t\in[N,2N] \,.
\end{equation}
Bearing in mind \eqref{Lambdaomega} -- \eqref{Phiest2} we obtain
\begin{equation}\label{Lambdaomegaest}
\sum\limits_{N<n\leq 2N}\Lambda(n)\omega(n,h)\ll|h|N^{\gamma-1}\max\limits_{N_1\in[N, 2N]}|\Sigma_4(N, N_1)|\,.
\end{equation}
Thus from \eqref{Sigmaest2}, \eqref{ahbh}, \eqref{Sigma2est2}, \eqref{Sigma3} and \eqref{Lambdaomegaest} it follows
\begin{equation}\label{Sigmaest3}
\Sigma\ll (\log x)\max_{\Delta_1^c \tan^\theta(\log \Delta_1)-1
\leq N\leq \frac{1}{2}\Delta_2^c \tan^\theta(\log \Delta_2)}\Big(N^{\gamma-1} \Sigma_5(N, N_1)\Big)+\frac{x}{\log^2x}\,,
\end{equation}
where
\begin{equation}\label{Sigma5}
\Sigma_5(N, N_1)=\sum\limits_{1\leq h\leq M}\max\limits_{N_1\in[N, 2N]}|\Sigma_4(N, N_1)|\,.
\end{equation}
It remains to apply Heath-Brown's identity to the sum  $\Sigma_4(N, N_1)$ defined by \eqref{Sigma4}.

\begin{lemma}\label{SIest}
Set
\begin{equation}\label{SI}
S_I=\mathop{\sum\limits_{D<d\le D_1}a(d)\sum\limits_{L<l\le L_1}}_{N<dl\le N_1}e\big(hm'_{dl}\big)
\end{equation} and
\begin{equation}\label{SI'}
S'_I=\mathop{\sum\limits_{D<d\le D_1}a(d)\sum\limits_{L<l\le L_1}}_{N<dl\le N_1}e\big(hm'_{dl}\big)\log l\,,
\end{equation}
where
\begin{equation}\label{Conditions1}
L\ge 2^{-10}N^{\frac{1}{2}}\,,\quad D_1\le 2D\,,\quad L_1\le 2L\,,\quad N_1\le 2N\,,\quad a(d)\ll \tau _5(d)\log N\,.
\end{equation}
Then
\begin{equation*}
S_I,\, S'_I\ll h^{\frac{1}{6}}N^{\frac{2\gamma+9}{12}+\varepsilon} \,.
\end{equation*}
\end{lemma}
\begin{proof}
First we notice that \eqref{SI}  and \eqref{Conditions1} imply
\begin{equation}\label{LMasympX}
DL\asymp N\,.
\end{equation}
Denote
\begin{equation}\label{flm}
f(d, l)=m'_{dl}\,.
\end{equation}
By   \eqref{SI}, \eqref{Conditions1} and \eqref{flm}  we write
\begin{equation}\label{SIest1}
S_I\ll N^\varepsilon\sum\limits_{D<d\le D_1}\bigg|\sum\limits_{L'<l\leq L'_1}e\big(h f(d, l)\big)\bigg|\,,
\end{equation}
where
\begin{equation}\label{L'L1'}
L'=\max{\bigg\{L,\frac{N}{m}\bigg\}}\,,\quad L_1'=\min{\bigg\{L_1,\frac{N_1}{m}\bigg\}}\,.
\end{equation}
From  \eqref{Conditions1} and \eqref{L'L1'} for the sum in \eqref{SIest1} it follows
\begin{equation}\label{L'andL1'inL2L}
\left|\begin{array}{cccc}
D<d\le D_1\quad \quad\\
L'<l\leq L'_1  \quad \quad\; \\
N<dl\le N_1 \quad \;\;\\
(L', L'_1]\subseteq (L, 2L]
\end{array}\right..
\end{equation}
According to \eqref{subset}, \eqref{System1}, \eqref{NXc}, \eqref{Conditions1}, \eqref{flm} and \eqref{L'andL1'inL2L}
the function $f(d, l)$ is an implicit function of $d$ and $l$ defined by
\begin{equation}\label{Implicitfunction3}
y^c\tan^\theta(\log y)=dl \quad \mbox{ for } \quad  dl\in (N, N_1]
\end{equation}
and
\begin{equation}\label{yDelta1Delta2SI}
y\subset[\Delta_1, \Delta_2)\,.
\end{equation}
On the other hand for the function defined by  \eqref{flm} and \eqref{Implicitfunction3} we find
\begin{equation}\label{firstderivativel}
\frac{\partial f(d, l)}{\partial l}=\frac{dy^{1-c}}{\big(c\tan(\log y)+\theta\sec^2(\log y)\big)\tan^{\theta-1}(\log y)}\,.
\end{equation}
Now \eqref{Delta1},  \eqref{Delta2}, \eqref{NXc}, \eqref{LMasympX}, \eqref{L'andL1'inL2L},
\eqref{yDelta1Delta2SI} and \eqref{firstderivativel}  yields
\begin{equation}\label{firstderivativelasymp}
\frac{\partial f(d,l)}{\partial l}\asymp N^\gamma L^{-1}\,.
\end{equation}
Proceeding in the same way we get
\begin{equation}\label{secondderivativelasymp}
\frac{\partial^2f(d,l)}{\partial l^2}\asymp N^\gamma L^{-2}
\end{equation}
and
\begin{equation}\label{thirdderivativelasymp}
\frac{\partial^3f(d,l)}{\partial l^3}\asymp  N^\gamma L^{-3}\,.
\end{equation}
Using  \eqref{Conditions1}, \eqref{LMasympX}, \eqref{SIest1}, \eqref{L'andL1'inL2L}, \eqref{firstderivativelasymp},
\eqref{secondderivativelasymp}, \eqref{thirdderivativelasymp} and Lemma \ref{GrahamandKolesnik} with  $k=1$ we obtain
\begin{align*}\label{SIest1}
S_I&\ll N^\varepsilon\sum\limits_{D<d\le D_1}\Big(h^{\frac{1}{6}}N^{\frac{\gamma}{6}} L^{\frac{1}{2}}
+h^{-1}N^{-\gamma}L\Big)\\
&\ll N^\varepsilon\Big(D h^{\frac{1}{6}}N^{\frac{\gamma}{6}} L^{\frac{1}{2}} +h^{-1}N^{-\gamma}DL\Big)\\
&\ll N^\varepsilon\Big(D^{\frac{1}{2}} h^{\frac{1}{6}}N^{\frac{1}{2}+\frac{\gamma}{6}}  +h^{-1}N^{1-\gamma}\Big)\\
&\ll N^\varepsilon\Big( h^{\frac{1}{6}}N^{\frac{3}{4}+\frac{\gamma}{6}}  +h^{-1}N^{1-\gamma}\Big)\\
&\ll h^{\frac{1}{6}}N^{\frac{2\gamma+9}{12}+\varepsilon} \,.                                                                                                                     \end{align*}
To estimate the sum defined by \eqref{SI'} we apply Abel's summation formula and proceed in the same way to deduce
\begin{equation*}
S_I'\ll h^{\frac{1}{6}}N^{\frac{2\gamma+9}{12}+\varepsilon} \,.
\end{equation*}
This proves the lemma.
\end{proof}

\begin{lemma}\label{SIIest}
Set
\begin{equation}\label{SII}
S_{II}=\mathop{\sum\limits_{D<d\le D_1}a(d)\sum\limits_{L<l\le L_1}}_{N<dl\le N_1}b(l)e\big(hm'_{dl}\big)\,,
\end{equation}
where
\begin{equation}
\begin{split}\label{Conditions2}
&2^3\leq L\leq 2^7N^{\frac{1}{3}}\,,\quad D_1\le 2D\,,\quad L_1\le 2L\,,\quad N_1\le 2N\,,\\
&a(d)\ll \tau _5(d)\log N\,,\quad b(l)\ll \tau _5(l)\log N\,.
\end{split}
\end{equation}
Then
\begin{equation*}
S_{II}\ll h^{\frac{1}{4}} N^{\frac{3\gamma+8}{12}+\varepsilon}\,.
\end{equation*}
\end{lemma}
\begin{proof}
First we notice that  \eqref{SII}  and \eqref{Conditions2} give us
\begin{equation}\label{LMasympX2}
DL\asymp N\,.
\end{equation}
From \eqref{Delta1}, \eqref{flm}, \eqref{SII}, \eqref{Conditions2}, \eqref{LMasympX2}, Cauchy's inequality
and Lemma \ref{Squareoutlemma} with $Q=N^{\frac{1}{2}}$ it follows
\begin{align}\label{SIIest1}
S_{II}&\ll\left(\sum\limits_{D<d\le D_1}|a(d)|^2\right)^{\frac{1}{2}}
\left(\sum\limits_{D<d\le D_1}\bigg|\sum\limits_{L<l\le L_1\atop{N<dl\le N_1}}
b(l)e\big(h f(d, l)\big)\bigg|^2\right)^{\frac{1}{2}}\nonumber\\
&\ll D^{\frac{1}{2}+\varepsilon}\left(\sum\limits_{D<d\le D_1}\frac{L}{Q}\sum_{|q|<Q}\bigg(1-\frac{q}{Q}\bigg)
\sum\limits_{L<l, \, l+q\leq L_1\atop{N<dl\le N_1\atop{N<d(l+q)\le N_1}}}b(l+q)\overline{b(l)}
e\Big(h f(d, l+q)-h f(d, l)\Big)\right)^{\frac{1}{2}}\nonumber\\
&\ll D^{\frac{1}{2}+\varepsilon}\Bigg(\frac{L}{Q}\sum\limits_{D<d\le D_1}\Bigg(L^{1+\varepsilon}\nonumber\\
&\hspace{20mm}+\sum_{1\leq |q|<Q}\bigg(1-\frac{q}{Q}\bigg)
\sum\limits_{L<l, \, l+q\leq L_1\atop{N<dl\le N_1\atop{N<d(l+q)\le N_1}}}
b(l+q)\overline{b(l)}e\Big(h f(d, l+q)-h f(d, l)\Big)\Bigg)^{\frac{1}{2}}\nonumber\\
&\ll N^\varepsilon\Bigg(\frac{N^2}{Q}+\frac{N}{Q}\sum\limits_{1\leq |q|\leq Q}
\sum\limits_{L<l, \, l+q\leq L_1}\bigg|\sum\limits_{D'<d\leq D_1'}e\big(h f(d, l, q)\big)\bigg|\Bigg)^{\frac{1}{2}}\,,
\end{align}
where
\begin{equation}\label{M'M1'}
D'=\max{\bigg\{D,\frac{N}{l},\frac{N}{l+q}\bigg\}}\,,
\quad D_1'=\min{\bigg\{D_1,\frac{N_1}{l},\frac{N_1}{l+q}\bigg\}} \end{equation}
and
\begin{equation}\label{fdlq}
f(d, l, q)=m'_{d(l+q)}-m'_{dl}\,.
\end{equation}
From  \eqref{Conditions2} and \eqref{M'M1'} for the sum  in \eqref{SIIest1} we have
\begin{equation}\label{M'M1'inM2M}
\left|\begin{array}{ccccc}
L<l, \, l+q\leq L_1\quad\; \\
D'<d\leq D_1'\quad \quad \quad \; \\
N<dl\le N_1\quad\quad\quad \; \\
N<d(l+q)\le N_1\;\; \\
(D', D'_1]\subseteq (D, 2D] \quad
\end{array}\right..
\end{equation}
According to \eqref{subset}, \eqref{System1}, \eqref{NXc}, \eqref{Conditions2}, \eqref{fdlq} and \eqref{M'M1'inM2M}
the function $f(d, l, q)$  is the difference of two implicit functions of $d$ and $l $ defined by
\begin{equation}\label{Implicitfunction4}
y^c\tan^\theta(\log y)=d(l+q) \quad \mbox{ for } \quad  d(l+q)\in (N, N_1]
\end{equation}
and
\begin{equation}\label{Implicitfunction5}
y^c\tan^\theta(\log y)=dl \quad \mbox{ for } \quad  dl\in (N, N_1]
\end{equation}
and for both functions
\begin{equation}\label{yDelta1Delta2SII}
y\subset[\Delta_1, \Delta_2)\,.
\end{equation}
We have
\begin{align}\label{firstderivatived}
\frac{\partial f(d, l, q)}{\partial d}
&=\frac{(l+q)y^{1-c}}{\big(c\tan(\log y)+\theta\sec^2(\log y)\big)\tan^{\theta-1}(\log y)}  \nonumber\\
&-\frac{ly^{1-c}}{\big(c\tan(\log y)+\theta\sec^2(\log y)\big)\tan^{\theta-1}(\log y)}
\end{align}
Now \eqref{Delta1},  \eqref{Delta2}, \eqref{NXc}, \eqref{LMasympX2}, \eqref{M'M1'inM2M},
\eqref{yDelta1Delta2SII} and \eqref{firstderivatived}  imply
\begin{equation}\label{firstderivativedasymp}
\frac{\partial f(d, l, q)}{\partial d}\asymp N^\gamma D^{-1}\,.
\end{equation}
Arguing in the same way we get
\begin{equation}\label{secondderivativedasymp}
\frac{\partial^2f(d,l,q)}{\partial d^2}\asymp N^\gamma D^{-2}\,. \end{equation}
Now \eqref{Conditions2}, \eqref{LMasympX2}, \eqref{SIIest1}, \eqref{M'M1'inM2M}, \eqref{firstderivativedasymp},
\eqref{secondderivativedasymp} and Lemma \ref{GrahamandKolesnik} with  $k=0$  give us
\begin{align*}
S_{II}&\ll N^\varepsilon\Bigg(\frac{N^2}{Q}+\frac{N}{Q}\sum\limits_{1\leq q\leq Q}
\sum\limits_{L<l\leq L_1}\Big(h^{\frac{1}{2}}N^{\frac{\gamma}{2}}+h^{-1}N^{-\gamma}D\Big)\Bigg)^{\frac{1}{2}}\\
&\ll N^\varepsilon\Big(N^2Q^{-1}+h^{\frac{1}{2}} L N^{1+\frac{\gamma}{2}}+h^{-1}N^{2-\gamma}\Big)^{\frac{1}{2}}\\
&\ll N^\varepsilon\Big(NQ^{-\frac{1}{2}}+h^{\frac{1}{4}} L^\frac{1}{2} N^{\frac{\gamma+2}{4}}+h^{-\frac{1}{2}}N^{1-\frac{\gamma}{2}}\Big)\\
&\ll N^\varepsilon\Big(N^{\frac{3}{4}}+h^{\frac{1}{4}} N^{\frac{3\gamma+8}{12}} +h^{-\frac{1}{2}}N^{1-\frac{\gamma}{2}}\Big) \\
&\ll h^{\frac{1}{4}} N^{\frac{3\gamma+8}{12}+\varepsilon}\,.
\end{align*}
This proves the lemma.
\end{proof}

\begin{lemma}\label{Sigma4est}
Then  for the exponential sum denoted by \eqref{Sigma4} we have
\begin{equation*}
\Sigma_4(N, N_1)\ll N^\varepsilon\Big(h^{\frac{1}{6}}N^{\frac{2\gamma+9}{12}}+ h^{\frac{1}{4}} N^{\frac{3\gamma+8}{12}}\Big)\,.
\end{equation*}
\end{lemma}
\begin{proof}
Take
\begin{equation}\label{UVZ}
U=2^3\,,\quad V=2^7N^{\frac{1}{3}}\,,\quad Z= 2^{-10}N^{\frac{1}{2}}\,.
\end{equation}
According to Lemma \ref{Heath-Brown}, the sum $\Sigma_4(N, N_1)$
can be decomposed into $O\Big(\log^6N\Big)$ sums, each of which is either of Type I
\begin{equation*}
S_I=\mathop{\sum\limits_{D<d\le D_1}a(d)\sum\limits_{L<l\le L_1}}_{N<dl\le N_1}e\big(hm'_{dl}\big)
\end{equation*} and
\begin{equation*}
S'_I=\mathop{\sum\limits_{D<d\le D_1}a(d)\sum\limits_{L<l\le L_1}}_{N<dl\le N_1}e\big(hm'_{dl}\big)\log l\,,
\end{equation*}
where
\begin{equation*}
L\ge Z\,,\quad D_1\le 2D\,,\quad L_1\le 2L\,,\quad a(d)\ll \tau _5(d)\log N
\end{equation*}
or of Type II
\begin{equation*}
S_{II}=\mathop{\sum\limits_{D<d\le D_1}a(d)\sum\limits_{L<l\le L_1}}_{N<dl\le N_1}b(l)e\big(hm'_{dl}\big)
\end{equation*}
where
\begin{equation*}
U\le L\le V\,,\quad D_1\le 2D\,,\quad L_1\le 2L\,,\quad
a(d)\ll \tau _5(d)\log N\,,\quad b(l)\ll \tau _5(l)\log N\,.
\end{equation*}
Bearing in mind  \eqref{UVZ}, Lemma \ref{SIest} and  Lemma \ref{SIIest} we establish the statement in the lemma.
\end{proof}
We are now in a good position to estimate the sum $\Sigma_5(N, N_1)$ defined by  \eqref{Sigma5}.
Using \eqref{M}, \eqref{Sigma5}  and Lemma \ref{Sigma4est} we write
\begin{align}\label{Sigma5est}
\Sigma_5(N, N_1)&\ll N^{\frac{2\gamma+9}{12}+\varepsilon}\sum\limits_{1\leq h\leq M}h^{\frac{1}{6}}
+ N^{\frac{3\gamma+8}{12}+\varepsilon}\sum\limits_{1\leq h\leq M}h^{\frac{1}{4}}\nonumber\\
&\ll N^{\frac{2\gamma+9}{12}+\varepsilon}M^{\frac{7}{6}}+ N^{\frac{3\gamma+8}{12}+\varepsilon}M^{\frac{5}{4}}\nonumber\\
&\ll N^{\frac{23-12\gamma}{12}+\varepsilon}\,.
\end{align}
We will now take the final step in estimating the sum $\Sigma$.
By \eqref{NXc}, \eqref{Sigmaest3} and \eqref{Sigma5est} we deduce
\begin{align}\label{Sigmaest4}
\Sigma&\ll (\log x)\max_{\Delta_1^c \tan^\theta(\log \Delta_1)-1
\leq N\leq \frac{1}{2}\Delta_2^c \tan^\theta(\log \Delta_2)}\Big(N^{\gamma-1}
N^{\frac{23-12\gamma}{12}+\varepsilon}\Big)+\frac{x}{\log^2x}\nonumber\\
&\ll (\log x)\max_{\Delta_1^c \tan^\theta(\log \Delta_1)-1
\leq N\leq \frac{1}{2}\Delta_2^c \tan^\theta(\log \Delta_2)}\Big(N^{\frac{11}{12}+\varepsilon}\Big)+\frac{x}{\log^2x}\nonumber\\
&\ll\frac{x}{\log^2 x}\,.
\end{align}

\subsection{The end of the proof}

Bearing in mind \eqref{ScxGammaSigma}, \eqref{Gammaest2} and \eqref{Sigmaest4} we establish that
\begin{equation*}
S_c(x)\gg \frac{x}{\log x}\,.
\end{equation*}

The Theorem is proved.

\vskip20pt
\footnotesize
\begin{flushleft}
S. I. Dimitrov\\
Faculty of Applied Mathematics and Informatics\\
Technical University of Sofia \\
8, St.Kliment Ohridski Blvd. \\
1756 Sofia, BULGARIA\\
e-mail: sdimitrov@tu-sofia.bg\\
\end{flushleft}
\end{document}